\documentclass[10pt]{amsart}

 \usepackage{amssymb,latexsym,amsmath,amsthm}

\usepackage{fullpage}

\usepackage[numbers,sort&compress]{natbib}

\numberwithin{equation}{section}

\newtheorem*{thm*}{Theorem A}

\newtheorem{thm}{Theorem}[section]

\newtheorem{lemma}{Lemma}[section]

\newtheorem{remark}{Remark}[section]

\newtheorem*{acknow}{Acknowlegement}

\begin{document}

\def\d{ \partial_{x_j} }
\def\Na{{\mathbb{N}}}

\def\Z{{\mathbb{Z}}}

\def\IR{{\mathbb{R}}}

\newcommand{\E}[0]{ \varepsilon}

\newcommand{\la}[0]{ \lambda}

\newcommand{\s}[0]{ \mathcal{S}}

\newcommand{\AO}[1]{\| #1 \| }

\newcommand{\BO}[2]{ \left( #1 , #2 \right) }

\newcommand{\CO}[2]{ \left\langle #1 , #2 \right\rangle}

\newcommand{\R}[0]{ \IR\cup \{\infty \} }

\newcommand{\co}[1]{ #1^{\prime}}

\newcommand{\p}[0]{ p^{\prime}}

\newcommand{\m}[1]{   \mathcal{ #1 }}

\newcommand{ \W}[0]{ \mathcal{W}}

\newcommand{ \A}[1]{ \left\| #1 \right\|_H }

\newcommand{\B}[2]{ \left( #1 , #2 \right)_H }

\newcommand{\C}[2]{ \left\langle #1 , #2 \right\rangle_{  H^* , H } }

 \newcommand{\HON}[1]{ \| #1 \|_{ H^1} }

\newcommand{ \Om }{ \Omega}

\newcommand{ \pOm}{\partial \Omega}

\newcommand{\D}{ \mathcal{D} \left( \Omega \right)}

\newcommand{\DP}{ \mathcal{D}^{\prime} \left( \Omega \right)  }

\newcommand{\DPP}[2]{   \left\langle #1 , #2 \right\rangle_{  \mathcal{D}^{\prime}, \mathcal{D} }}

\newcommand{\PHH}[2]{    \left\langle #1 , #2 \right\rangle_{    \left(H^1 \right)^*  ,  H^1   }    }

\newcommand{\PHO}[2]{  \left\langle #1 , #2 \right\rangle_{  H^{-1}  , H_0^1  }}

 \newcommand{\HO}{ H^1 \left( \Omega \right)}

\newcommand{\HOO}{ H_0^1 \left( \Omega \right) }

\newcommand{\CC}{C_c^\infty\left(\Omega \right) }

\newcommand{\N}[1]{ \left\| #1\right\|_{ H_0^1  }  }

\newcommand{\IN}[2]{ \left(#1,#2\right)_{  H_0^1} }

\newcommand{\INI}[2]{ \left( #1 ,#2 \right)_ { H^1}}

\newcommand{\HH}{   H^1 \left( \Omega \right)^* }

\newcommand{\HL}{ H^{-1} \left( \Omega \right) }

\newcommand{\HS}[1]{ \| #1 \|_{H^*}}

\newcommand{\HSI}[2]{ \left( #1 , #2 \right)_{ H^*}}

\newcommand{\WO}{ W_0^{1,p}}
\newcommand{\w}[1]{ \| #1 \|_{W_0^{1,p}}}

\newcommand{\ww}{(W_0^{1,p})^*}

\newcommand{\Ov}{ \overline{\Omega}}

\title{Regularity of the extremal solutions associated to elliptic systems}

\author{Craig Cowan}
\author{Mostafa Fazly}

\address{Department of Mathematical Sciences, University of Alabama in Huntsville, 258A Shelby Center, Huntsville, AL 35899.}

\email{ctc0013@uah.edu}

\address{Department of Mathematical and Statistical Sciences, University of Alberta, Edmonton, Alberta, Canada T6G 2G1}

\email{fazly@ualberta.ca}

\thanks{The second author  is pleased to acknowledge the support of a University of Alberta start-up grant RES0019810.}

\maketitle

\vspace{3mm}

\begin{abstract}  We examine the two elliptic systems given by
\[ (G)_{\lambda,\gamma} \quad    -\Delta u    = \lambda f'(u) g(v),  \quad -\Delta v = \gamma f(u) g'(v)  \quad \mbox{ in $ \Omega$},  \]  and
\[ (H)_{\lambda,\gamma} \quad    -\Delta u    = \lambda f(u) g'(v),  \quad -\Delta v = \gamma f'(u) g(v)  \quad \mbox{ in $ \Omega$},\] with zero Dirichlet boundary conditions and where    $ \lambda,\gamma$ are positive parameters.    We show that for general nonlinearities $f$ and $g$ the extremal solutions associated with $ (G)_{\lambda,\gamma}$ are bounded, provided $ \Omega$ is a convex domain in $ \IR^N$ where $ N \le 3$.   In the case of a radial domain, we show the extremal solutions are bounded provided  $ N <10$.    The extremal solutions associated with $ (H)_{\lambda,\gamma}$ are bounded in the case where $ f$ is a general nonlinearity, $ g(v)=(v+1)^q$ for $ 1 <q<\infty$ and when $ \Omega$ is a bounded convex domain in $ \IR^N$ for $ N \le 3$.  Certain regularity results are also obtained in higher dimensions for $ (G)_{\lambda,\gamma}$ and $(H)_{\lambda,\gamma}$  for the case of explicit nonlinearities of the form $ f(u)=(u+1)^p$ and $ g(v)=(v+1)^q$.

\end{abstract}

\noindent
{\it \footnotesize 2010 Mathematics Subject Classification}. {\scriptsize }

\section{Introduction}

We examine the following  systems:

 \begin{eqnarray*}
(G)_{\lambda,\gamma}\qquad  \left\{ \begin{array}{lcl}
\hfill   -\Delta u    &=& \lambda f'(u) g(v) \qquad \Omega  \\
\hfill -\Delta v &=& \gamma f(u) g'(v)  \qquad \Omega,  \\
\hfill u &=& v =0 \qquad \qquad \pOm
\end{array}\right.
  \end{eqnarray*}   and

 \begin{eqnarray*}
(H)_{\lambda,\gamma}\qquad  \left\{ \begin{array}{lcl}
\hfill   -\Delta u    &=& \lambda f(u) g'(v) \qquad \Omega  \\
\hfill -\Delta v &=& \gamma f'(u) g(v)  \qquad \Omega ,  \\
\hfill u &=& v =0 \qquad \qquad \pOm
\end{array}\right.
  \end{eqnarray*}  where $\Omega$ is a bounded domain in $\IR^N$ and $ \lambda, \gamma >0$ are positive parameters.  The nonlinearities $f$ and $g$ will satisfy various properties but will always at least satisfy
 \[ (R) \;\;  \mbox{$f$ is smooth, increasing and convex with $ f(0)=1 $ and $ f$ superlinear at $ \infty$.}\]
  We begin by recalling the scalar analog of the above systems.  Given a  nonlinearity $ f$ which satisfies (R), the following equation
 \begin{eqnarray*}
\hbox{$(Q)_{\lambda}$}\hskip 50pt \left\{ \begin{array}{lcl}
\hfill   -\Delta u  &=& \lambda f(u) \qquad \Omega \\
\hfill u&=& 0 \qquad\qquad \pOm
\end{array}\right.
  \end{eqnarray*}
  is  now quite well understood whenever $ \Omega$ is a bounded smooth domain in $ \IR^N$. See, for instance, \cite{BV,Cabre,CC, CR, MP,Nedev,bcmr}. We now list the  properties one comes to expect when studying $(Q)_{\lambda}$.   It is well known that there exists a critical parameter  $ \lambda^* \in (0,\infty)$, called the extremal parameter,  such that for all $ 0<\lambda < \lambda^*$ there exists a smooth, minimal solution $u_\lambda$ of $(Q)_\lambda$.  Here the minimal solution means in the pointwise sense.  In addition for each $ x \in \Omega$ the map $ \lambda \mapsto u_\lambda(x)$ is increasing in $ (0,\lambda^*)$.   This allows one to define the pointwise limit $ u^*(x):= \lim_{\lambda \nearrow \lambda^*} u_\lambda(x)$  which can be shown to be a weak solution, in a suitably defined sense, of $(Q)_{\lambda^*}$.  For this reason $ u^*$ is called the extremal solution.
   It is also known that for $ \lambda >\lambda^*$ there are no weak solutions of $(Q)_\lambda$.  Also one can show the minimal solution $ u_\lambda$  is a semi-stable solution of $(Q)_\lambda$  in the sense that
   \[ \int_\Omega \lambda  f'(u_\lambda) \psi^2 \le \int_\Omega | \nabla \psi|^2, \qquad \forall \; \psi \in H_0^1(\Omega).\]
   A question that has attracted a lot of attention is the regularity of the extremal solution.  It is known that the extremal solution can be a classical solution or it can be a singular weak solution.   We now list some results in this direction:
 \begin{itemize}
   \item (\cite{Nedev})  $u^*$ is bounded if $f$ satisfies (R) and $ N \le 3$.

   \item (\cite{Cabre}) $u^*$ is bounded if $f$ satisfies (R) (can drop the convexity assumption) and $ \Omega$ a convex domain in $ \IR^4$.

   \item (\cite{CC}) $u^*$ is bounded if $ \Omega$ is a radial domain in $ \IR^N$ with $ N <10$ and $ f$ satisfies (R) (can drop the convexity assumption).

   \end{itemize}

 It is precisely these type of results which we are interested in extending to the case of systems.  Before we can discuss the regularity of the extremal solutions associated with $(G)_{\lambda,\gamma}$ and $ (H)_{\lambda,\gamma}$ we need to introduce some notation.

Under various conditions on $f$ and $g$ the above systems fit into the general framework of developed in \cite{Mont}, who examined  a  generalization of
\begin{eqnarray*}
(P)_{\lambda,\gamma}\qquad  \left\{ \begin{array}{lcl}
\hfill   -\Delta u    &=& \lambda F(u,v) \qquad \Omega  \\
\hfill -\Delta v &=& \gamma G(u,v)   \qquad \Omega ,  \\
\hfill u &=& v =0 \qquad \qquad \pOm.
\end{array}\right.
  \end{eqnarray*}
The following results are all taken from \cite{Mont}. Let   $ \mathcal{Q}=\{ (\lambda,\gamma): \lambda, \gamma >0 \}$ and we define
\[ \mathcal{U}:= \left\{ (\lambda,\gamma) \in \mathcal{Q}: \mbox{ there exists a smooth solution $(u,v)$ of $(P)_{\lambda,\gamma}$} \right\}.\]  Firstly we assume that $F(0,0),G(0,0)>0$.
 A simple argument shows that if  $F$ is superlinear at $ u=\infty$, uniformly in $v$,  then the set of $ \lambda$ in $\mathcal{U}$ is bounded.  Similarly we assume that $ G$ is superlinear at $ v=\infty$, uniformly in $u$ and hence we get $ \mathcal{U}$ is bounded.
 We also assume that $F,G$ are increasing in each variable.   This allows the use of a sub/supersolution approach and one easily sees that if $ (\lambda,\gamma) \in \mathcal{U}$ then so is $ (0,\lambda] \times (0,\gamma]$.   One also sees that $ \mathcal{U}$ is nonempty.

We now define
 $ \Upsilon:= \partial \mathcal{U} \cap \mathcal{Q}$, which plays the role of the extremal parameter $ \lambda^*$.      Various properties of $ \Upsilon$ are known, see \cite{Mont}.     Given $ (\lambda^*,\gamma^*) \in \Upsilon$ set $ \sigma:= \frac{\gamma^*}{\lambda^*} \in (0,\infty)$ and define
 \[ \Gamma_\sigma:=\{ (\lambda, \lambda \sigma):  \frac{\lambda^*}{2} < \lambda < \lambda^*\}.\]   We let $ (u_\lambda,v_\lambda)$ denote the minimal solution $(P)_{\lambda, \sigma \lambda}$ for $ \frac{\lambda^*}{2} < \lambda < \lambda^*$.  One easily sees that for each $ x \in \Omega$ that $u_\lambda(x), v_\lambda(x)$ are increasing in $ \lambda$ and hence we define
 \[ u^*(x):= \lim_{\lambda \nearrow \lambda^*} u_\lambda(x), \quad  v^*(x):= \lim_{\lambda \nearrow \lambda^*} v_\lambda(x),\]  and we call $(u^*,v^*)$ the extremal solution associated with $ (\lambda^*,\gamma^*) \in \Upsilon$.
 Under some very  minor growth assumptions on $F$ and $G$ one can show that $(u^*,v^*)$ is a weak solution of $(P)_{\lambda^*,\gamma^*}$.

 We now come to the issue of stability.
\begin{thm*} \label{stable} \cite{Mont}   Let $ (\lambda,\gamma) \in \mathcal{U} $ and let $ (u,v)$ denote the minimal solution of $(P)_{\lambda,\gamma}$.  Then $(u,v)$ is semi-stable in the sense that there is some smooth $ 0 < \zeta,\chi \in H_0^1(\Omega)$ and $  0 \le \eta $  such that
\begin{equation} \label{sta}
 -\Delta \zeta = \lambda  F_u(u,v) \zeta + \lambda F_v(u,v) \chi + \eta \zeta, \qquad -\Delta \chi = \gamma G_u(u,v) \zeta + \gamma G_v(u,v) \chi + \eta \chi, \qquad \Omega.
 \end{equation}
\end{thm*}
In this paper we prove that the extremal solution of $(G)_{\lambda^*,\gamma^*}$ with general nonlinearities, either on a general domain and lower dimensions or on a radial domain and higher dimensions are regular. Moreover, for explicit nonlinearities we prove regularity on a general domain in higher dimensions.

The following stability inequalities play a key role in this paper and we shall refer to them many times through proofs. We mention that in \cite{fg} the De Giorgi type results and Liouville theorems have been proved for a much more general gradient system and they obtained a stability inequality which reduces to (\ref{gra}) in the particular case we are examining. Note that some of our results will hold for a general gradient system that is examined in \cite{fg}.

\begin{lemma} \label{stabb} For any $\phi,\psi \in H_0^1(\Omega)$ the following inequalities hold.
 \begin{enumerate} \item Let $(u,v)$ denote a semi-stable solution of $(G)_{\lambda,\gamma}$ in the sense of (\ref{sta}).  Then
\begin{equation} \label{gra}
\int f''(u) g(v) \phi^2 + \int f(u) g''(v) \psi^2 +2 \int f'(u) g'(v) \phi \psi \le \frac{1}{\lambda} \int | \nabla \phi|^2 + \frac{1}{\gamma} \int | \nabla \psi|^2.
\end{equation}

\item Let $(u,v)$ denote a semi-stable solution of $(H)_{\lambda,\gamma}$ in the sense of (\ref{sta}). Then
\begin{equation} \label{twist}
 \int f'(u) g'(v) (\phi^2+ \psi^2)+ 2 \int \sqrt{f f'' g g''}  \phi \psi \le \frac{1}{\lambda} \int | \nabla \phi|^2 + \frac{1}{\gamma} \int | \nabla \psi|^2.
\end{equation}

\end{enumerate}

\end{lemma}

\begin{proof} We will prove inequalities (\ref{gra}) and (\ref{twist}) for $ \phi,\psi \in C_c^\infty(\Omega)$ and then a standard density argument extends the inequalities to $ \phi,\psi \in H_0^1(\Omega)$. \\
 (1)   By Theorem A  there is some $ 0 < \zeta,\chi$ such that
\[ -\Delta \zeta \ge \lambda f''(u) g(v) {\zeta}  + \lambda f'(u) g'(v) {\chi}  \quad \mbox{and} \quad {-\Delta \chi} \ge \gamma f'(u) g'(v) {\zeta} + \gamma f(u) g''(v) {\chi} \quad \mbox{in $ \Omega$} .\]
Consider test functions  $\phi,\psi \in C_c^\infty(\Omega)$ and multiply both sides of the above inequalities with $\frac{\phi^2}{\zeta}$ and $\frac{\psi^2}{\chi}$ to obtain
\begin{eqnarray*}
&&-\int  |\nabla\zeta|^2 \frac{\phi^2}{\zeta^2}+ 2 \int \nabla \phi\cdot\nabla \zeta\ \frac{\phi}{\zeta} \ge \int   \lambda f'(u) g'(v) \phi^2 \frac{\chi}{\zeta} +\int  \lambda f(u) g''(v) \phi^2 ,\\
&&-\int  |\nabla\chi|^2 \frac{\psi^2}{\chi^2}+ 2 \int \nabla \psi\cdot\nabla \chi\ \frac{\psi}{\chi} \ge  \int \gamma  f(u) g''(v) \psi^2 +\int \gamma f'(u) g'(v) \psi^2 \frac{\zeta}{\chi},
  \end{eqnarray*}  note there are no issues regarding the functions $\frac{\phi^2}{\zeta}, \frac{\psi^2}{\chi}$ after one considers the fact that $ \zeta$ and $ \chi$ are smooth and positive on the support of $ \phi$ and $ \psi$.
Apply Young's inequality for the left hand side of each inequality and add them to get
\begin{eqnarray*}
\lambda  \int f''(u) g(v) \phi^2 + \gamma \int  f(u) g''(v) \psi^2 + \int f'(u) g'(v) \left(\lambda \phi^2 \frac{\chi}{\zeta} + \gamma \psi^2 \frac{\zeta}{\chi} \right) \le \int |\nabla \phi|^2 +\int |\nabla \psi|^2.
  \end{eqnarray*}
Simple calculations show that the third term is an upper bound for
$$  2\sqrt{\lambda \gamma}   \int  f'(u) g'(v) \phi\psi.$$
Then, replacing $ \phi$ with $ \frac{\phi}{\sqrt{\lambda}}$ and $ \psi$ with $ \frac{\psi}{\sqrt{\gamma}}$ gives the desired result.

(2) Proof is quite similar to (1).   By Theorem A there is some $ 0 < \zeta,\chi$ such that
\[ -\frac{\Delta \zeta}{ \zeta} \ge \lambda f'(u) g'(v) + \lambda f(u) g''(v) \frac{\chi}{\zeta} \quad \mbox{and} \quad -\frac{\Delta \chi}{\chi} \ge \gamma f''(u) g(v) \frac{\zeta}{\chi} + \gamma f'(u) g'(v) \quad \mbox{in $ \Omega$}, \]  and we now multiply the first equation by $ \phi^2$ and the second by $ \psi^2$ and add the equations and integrate over $\Omega$. In addition we use the   fact that
\[ \int_\Omega \frac{-\Delta E}{E} \phi^2 \le \int | \nabla \phi|^2,\] for any $ E>0$ and $ \phi \in H_0^1(\Omega)$.  Doing this one obtains
\[\int_\Omega f'(u) g'(v) (\lambda \phi^2 + \gamma \psi^2) + \int_\Omega \lambda f(u) g''(v) \phi^2 \frac{\chi}{\zeta} + \gamma f''(u) g(v) \psi^2 \frac{\zeta}{\chi} \le \int_\Omega | \nabla \phi|^2 + | \nabla \psi|^2.\]  Again some simple algebra shows that
\[ 2 \sqrt{\lambda \gamma} \int_\Omega \sqrt{ f(u) f''(u) g(v) g''(v)} \phi \psi,\] is a lower bound for the second integral.  Using this lower bound and replacing $ \phi$ with $ \frac{\phi}{\sqrt{\lambda}}$ and $ \psi$ with $ \frac{\psi}{\sqrt{\gamma}}$ finishes the proof.

\end{proof}

  In Section 2, we explore the regularity of extremal solutions for systems  $(G)_{\lambda,\gamma}$ and $(H)_{\lambda,\gamma}$ with general nonlinearities and,  in then Section 3 we consider explicit nonlinearities.   We finish the current section by this point that in \cite{craig0} the system
 \begin{eqnarray*}
(E)_{\lambda,\gamma}\qquad  \left\{ \begin{array}{lcl}
\hfill   -\Delta u    &=& \lambda e^v \qquad \Omega  \\
\hfill -\Delta v &=& \gamma e^u  \qquad \Omega,  \\
\hfill u &=& v =0 \qquad  \pOm,
\end{array}\right.
  \end{eqnarray*} was examined.  It was shown that if $ \Omega$ is a bounded domain in $ \IR^N$ where $ N \le9$, then the extremal solution $(u^*,v^*)$ associated with $(\lambda^*,\gamma^*) \in \Upsilon$ is bounded if
  \[ \frac{N-2}{8} < \frac{\gamma^*}{\lambda^*} < \frac{8}{N-2}.\]    Note that as one gets closer to the diagonal parameter range  $ \gamma=\lambda$ that better regularity results are obtained.  At the diagonal the system can be shown to reduce to the scalar equation $ -\Delta u = \lambda e^u$.  This phenomena will also be present in Section 3 where we consider explicit nonlinearities.

  \begin{acknow}  The authors would like to thank the anonymous referee for many valuable suggestions. In addition in the original version of this paper we had stronger assumptions on $f$ and $g$ in Theorem \ref{nedev}.  The referee suggested this could probably be weakened to the current conditions.
  
  \end{acknow}

  \section{General nonlinearities}
We begin by examining $(G)_{\lambda,\gamma}$ in the case of general nonlinearities and we show the extremal solutions are bounded in low dimensions and our methods of proof are close to  \cite{Nedev}.
\begin{thm} \label{nedev} Suppose that $ \Omega$ is a bounded smooth convex domain in $ \IR^N$  where $N \le 3$ and suppose $f$ and $g$ both satisfy condition (R). We also assume that $ f'(0)>0$ and $ g'(0)>0$.  In addition we assume that $ f',g'$ are convex and there is some $ \xi >0$ such that
\begin{equation} \label{thing}
 \liminf_{u \rightarrow \infty} f''(u) =\infty, \qquad \liminf_{v \rightarrow \infty} {g''(v)}=\infty.
\end{equation}
Let $(\lambda^*,\gamma^*) \in \Upsilon$.  Then the associated extremal solution of  $ (G)_{\lambda^*,\gamma^*}$ is bounded.
\end{thm}
For radial domains we obtain similar results but in higher dimensions and our methods of proof follow very closely to  \cite{CC}  and \cite{v}.
\begin{thm} \label{radial}
 Let $\Omega = B_1$, $N \ge 3$, and  $f$ and $g$ both satisfy condition (R) and in addition we assume that there is some $ \xi>0$ such that
\[ \liminf_{u \rightarrow \infty} f''(u)=\infty, \quad \liminf_{v \rightarrow \infty} g''(v)=\infty. \]  Let $ (\lambda^*,\gamma^*) \in \Upsilon$ and let $ (u^*,v^*)$ denote the extremal solution associated with $ (G)_{\lambda^*,\gamma^*}$.  Then
 \begin{enumerate}
 \item if $N<10$, then $u^*, v^*\in L^\infty(B_1)$,
 \item if $N=10$, then $u^*(r),v^*(r) \le C_{\lambda^*,\gamma^*}  (1+|\log r|)$ for $r\in (0,1]$,
 \item if $N>10$, then $u^*(r),v^*(r)\le C_{\lambda^*,\gamma^*,N} r^{-\frac{N}{2}+\sqrt{N-1}+2} $ for $r\in (0,1]$.
 \end{enumerate}
\end{thm}
We are unable to prove the analogous version  for the system $(H)_{\lambda,\gamma}$ and hence we restrict our attention to the special case.
\begin{thm} \label{pppp}  Suppose $ \Omega$ a bounded smooth convex domain in $ \IR^3$ and $ 1 < q < \infty$. Assume $f$ satisfies (R) and we also assume that $ f'' \ge C>0$.  Let $(\lambda^*,\gamma^*) \in \Upsilon$.  Then the associated extremal solution of  $ (H)_{\lambda^*,\gamma^*}$ for $g(v)=(1+v)^q$ is bounded.
\end{thm}
We fix the notation $ t_+(q):=q+\sqrt{q(q-1)}$ which will play an important role in the proof of the above theorem and theorems to follow.
The following lemma is used to prove Theorem \ref{nedev} where a convex domain is assumed but we prove the lemma for general domains.
\begin{lemma} \label{nedev_L} Suppose $ \Omega$ is a bounded domain in $ \IR^N$ and $ f$ and $g$ satisfy the conditions from Theorem \ref{nedev} and define $a:=f'(0)>0, b:=g'(0)>0$.  Let $ (\lambda^*,\gamma^*) \in \Upsilon$ and let $ (u^*,v^*)$ denote the extremal solution associated with $ (G)_{\lambda^*,\gamma^*}$. Then there is some $  C < \infty$ such that
\[(i) \; \; \int f'(u^*) g'(v^*) (f'(u^*)-a) (g'(v^*)-b) \le C, \quad (ii) \; \; \int (f'(u^*)-a) f''(u^*) g(v^*) \le C, \] \[ \quad (iii) \; \; \int (g'(v^*)-b) g''(v^*) f(u^*) \le C.\]
\end{lemma}
\begin{remark}  \label{rem1}  Let $ f_i,g_i$ denote smooth increasing nonlinearities with $ f_i(0),g_i(0)>0$ and we also assume
\begin{equation} \label{condi}
\liminf_{u \rightarrow \infty} \frac{f_1(u)}{u} = \liminf_{v \rightarrow \infty} \frac{g_2(v)}{v} =\infty.
\end{equation}    Let $(u_m,v_m)$  denote a sequence of  smooth solutions  of  $(P)_{\lambda_m, \sigma \lambda_m}$,  where $ 0< \sigma < \infty$ is fixed and $ \lambda_m$ is restricted to a compact subset of $(0,\infty)$ and  $ F(u,v)=f_1(u) g_1(v)$, $ G(u,v)= f_2(u) g_2(v)$.  Then we have the estimate
\[ \int f_1(u_m) g_1(v_m) \delta  +  \int f_2(u_m) g_2(v_m) \delta \le C,\]  where $ \delta(x):=dist(x,\pOm)$.  Applying  regularity theory shows that $ u_m,v_m$ are bounded in $L^1(\Omega)$.

On occasion we will restrict our attention to smooth convex domains where many of the proofs are much more compact. One can use the  moving plane to obtain uniform estimates for arbitrary positive solutions of our system in a  convex domain; see \cite{t}.    One first uses the results from \cite{t} to obtain estimates valid near $ \pOm$.    Suppose $(u_m,v_m)$ are as above.  For $ \E>0$ small define $\Omega_\E:=\{ x \in \Omega: \delta(x) <\E \}$.   Using results from \cite{t} shows there is some small $ \E>0$ (depending only on $ \Omega$) and $ 0<C$ such that $ \sup_{\Omega_\E} u_m+ \sup_{\Omega_\E} v_m \le C \| u_m\|_{L^1(\Omega)} + C  \| v_m\|_{L^1(\Omega)} $ and since $u_m,v_m$ are bounded in $L^1(\Omega)$ we see that $u_m$ and $v_m$ are bounded in $ \Omega_\E$.   Using the maximum principle there is some $ C_1>0$ such that $ u_m, v_m \ge C_1$ in the compliment of $ \Omega_\E$.

\end{remark}

\begin{proof}  All integrals are over $ \Omega$ unless otherwise stated.  Our approach will be to obtain uniform estimates for any minimal solution $(u,v)$ of $(G)_{\lambda,\gamma}$ on the ray $ \Gamma_\sigma$ and then one sends $ \lambda \nearrow \lambda^*$ to obtain the same estimate for $ (u^*,v^*)$.   Let $(u,v)$ denote a smooth minimal solution of $(G)_{\lambda,\gamma}$ on the ray $\Gamma_\sigma$ and  put $ \phi:= f'(u)-a$ and $ \psi:=g'(v)-b$ into (\ref{gra}) to obtain
\begin{eqnarray}\label{zzz}
\nonumber&&\int f''(u) g(v) (f'(u)-a)^2 + \int f(u) g''(v) (g'(v)-b)^2 + 2 \int f'(u) g'(v) (f'(u)-a)(g'(v)-b) \\ &\le&
 \frac{1}{\lambda} \int \nabla (f'(u)-a) f''(u) \cdot \nabla u + \frac{1}{\gamma} \int \nabla (g'(v)-b) g''(v) \cdot  \nabla v.
\end{eqnarray}
Integrating the right-side of (\ref{zzz}) by parts shows that
\[  \int \nabla (f'(u)-a) f''(u) \cdot \nabla u =- \int (f'(u)-a) f'''(u) | \nabla u|^2 +\int (f'(u)-a) f''(u) (-\Delta u).\]
\[  \int \nabla (g'(v)-b) g''(v) \cdot \nabla v =- \int (g'(v)-b) g'''(v) | \nabla v|^2 +\int (g'(v)-b) g''(v) (-\Delta v).\]
In addition the other term in  (\ref{zzz}) involving $v$ is of the similar form.   We use the equation $(G)_{\lambda,\gamma}$ to replace $ -\Delta u$ and $ -\Delta v$ in the last equalities and simplify to get
 \begin{eqnarray*}
&&\frac{1}{\lambda} \int (f'(u)-a) f'''(u) | \nabla u|^2 + \frac{1}{\gamma} \int (g'(v)-b) g'''(v) | \nabla v|^2 + 2 \int f'(u)g'(v) (f'(u)-a)(g'(v)-b) \\& \le&
 a \int f''(u) (f'(u)-a) g(v) + b \int g''(v) (g'(v)-b) f(u).
 \end{eqnarray*}
   We now define
 $ h_1(u):= \int_0^u (f'(t)-a) f'''(t) dt$ and $ h_2(v):=\int_0^v (g'(t)-b) g'''(t) dt$.  Subbing this into the previous inequality and integrating by parts and using  $(G)_{\lambda,\gamma}$ again we arrive at
 \begin{eqnarray} \label{mm}
\nonumber&& \int h_1(u) f'(u) g(v) + \int h_2(v) f(u) g'(v)  + 2 \int f'(u) g'(v) (f'(u)-a) (g'(v)-b) \\ & \le & a \int f''(u) (f'(u)-a) g(v) +  b \int g''(v) (g'(v)-b) f(u).
 \end{eqnarray}  Now suppose $ u > \alpha >0$.  Then we have
 \[ h_1(u) \ge \int_\alpha^u (f'(t)-a) f'''(t) dt \ge (f'(\alpha)-a) ( f''(u)-f''(\alpha)),\] and so using the condition on $f''(u)$ we see that
 \[ \liminf_{u \rightarrow \infty} \frac{h_1(u)}{f''(u)} \ge f'(\alpha)-a,\] for any $ \alpha >0$.  But since $ f$ is convex and superlinear at infinity we see that $ \lim_{u \rightarrow \infty} \frac{h_1(u)}{f''(u)}=\infty$.  Similarly $ \lim_{v \rightarrow \infty} \frac{h_2(v)}{g''(v)}=\infty$.

  We now estimate the integral $ \int f''(u) g(v)(f'(u)-a)$.   There is some $ T>1$ large such that for all $ u \ge T$ we have $ h_1(u) \ge 100(a+1) f''(u)$ for all $ u \ge T$.  Then we have
  \[ \int f''(u) g(v) (f'(u)-a) = \int_{u \ge T} + \int_{u <T}  \le \frac{1}{100(a+1)} \int h_1(u) g(v) (f'(u)-a) + \int_{u<T} \int f''(u) g(v) (f'(u)-a).\]   We now estimate this last integral.    Let $ T$ be as above and fixed and we let $ k \ge 1$ denote a natural number.
   \[ \int_{u<T} f''(u) g(v) (f'(u)-a) = \int_{u<T, v<kT} + \int_{u<T, v \ge kT} = C(k,T) + \int_{u<T, v \ge kT} f''(u) g(v) (f'(u)-a) \]
  and we now estimate this last integral.     One easily sees that this last integral is bounded above by
  \[ \sup_{u<T} \frac{f''(u)}{f'(u)} \sup_{v >kT} \frac{g(v)}{(g'(v)-b) g'(v)} \int (f'(u)-a) (g'(v)-b) f'(u) g'(v).\] Combining this all together we see that for all sufficiently large $ T$ and all $1 \le k$ there is some constant $ C(k,T)$ such that
\begin{eqnarray*}
 \int f''(u) g(v) (f'(u)-a) &\le &\frac{1}{100(a+1)} \int h_1(u) f'(u) g(v) + C(k,T)  \\
 && + \sup_{u<T} \frac{f''(u)}{f'(u)} \sup_{v >kT} \frac{g(v)}{(g'(v)-b) g'(v)} \int (f'(u)-a) (g'(v)-b) f'(u) g'(v).
 \end{eqnarray*}
 Using the same argument  one can show for all sufficiently large $T$ and for all $ 1 \le k$ there is some $ C(k,T)$ such that
\begin{eqnarray*}
 \int g''(v) f(u) (g'(v)-b) &\le &\frac{1}{100(b+1)} \int h_2(v) g'(v) f(u) + C(k,T)  \\
 && + \sup_{v<T} \frac{g''(v)}{g'(v)} \sup_{u >kT} \frac{f(u)}{(f'(u)-a) f'(u)} \int (f'(u)-a) (g'(v)-b) f'(u) g'(v).
 \end{eqnarray*}
   Since $ f'',g'' \rightarrow \infty$ we see that
 \[ \lim_{k \rightarrow \infty}  \sup_{u >kT} \frac{f(u)}{(f'(u)-a) f'(u)} =0,\] and similarly for the other term.  Hence by taking $ k$ sufficiently large we can substitute everything back into (\ref{mm}) and see that all the integrals in (\ref{mm}) are bounded independent of $ \lambda$.

\end{proof}

\noindent\textbf{Proof of Theorem \ref{nedev}.}  We assume that $ N=3$ and $ \Omega$ is convex domain in $ \IR^3$.      The case of $ N=1,2$ is easier and we omit their proofs.     We suppose that $ (\lambda^*,\gamma^*) \in \Upsilon$ and $(u^*,v^*)$ is the associated extremal solution of $(G)_{\lambda^*,\gamma^*}$.  Set $ \sigma=\frac{\gamma^*}{\lambda^*}$.
Using Remark \ref{rem1}  along with Lemma \ref{nedev_L} we see that $ f'(u^*) g'(v^*) \in L^2(\Omega)$.   Note that this and the convexity of $ g$ show that
\[ \int_\Omega \frac{f'(u^*)^2 g(v^*)^2}{(v^*+1)^2} \le C.\]
From Lemma \ref{nedev_L} and Remark \ref{rem1} we have   $ -\Delta u^*, -\Delta v^* \in L^1$ and hence we have $ u^*,v^* \in L^{3_-}$, ie. $ L^p$ for any $ p <3$.   We now use the domain decomposition method as in \cite{Nedev}.  Set
\begin{eqnarray*}
 \Omega_1 &:=& \left\{ x: \frac{f'(u^*)^2 g(v^*)^2}{(v^*+1)^2} \ge f'(u^*)^{2-\alpha} g(v^*)^{2-\alpha} \right\},\\
\Omega_2 &:=& \Omega \backslash \Omega_1 = \left\{ x : f'(u^*) g(v^*) \le (v^*+1)^\frac{2}{\alpha} \right\},
 \end{eqnarray*}
where $  0 < \alpha  $ is to be picked later.  First note that
\[ \int_{\Omega_1} (f'(u^*) g(v^*))^{2-\alpha} \le \int_\Omega \frac{f'(u^*)^2 g(v^*)^2}{(v^*+1)^2} \le C.\] Similarly we have
\[ \int_{\Omega_2} (f'(u^*) g(v^*))^p \le \int_\Omega (v^*+1)^\frac{2p}{\alpha}.\]     Taking $ \alpha= \frac{4}{5}$ and using the $ L^{3_-}$ bound on $ v$ shows that $ f'(u^*) g(v^*) \in  L^{ \frac{6}{5}_-}(\Omega)$.   By a symmetry argument we also have $ f(u^*) g'(v^*) \in  L^{ \frac{6}{5}_-}(\Omega)$.

  By elliptic regularity we have $ u^*,v^* \in W^{2,\frac{6}{5}_-}$ and this is contained in $L^{6_-}(\Omega)$ after considering the Sobolev imbedding theorem.   Using these estimates and again using the domain decomposition $ \Omega_1$ and $ \Omega_2$ but taking $ \alpha=\frac{1}{2}$ gives that $ f'(u^*) g(v^*) \in L^{\frac{3}{2}_-}(\Omega)$ and by symmetry we have the same for $ f(u^*)g'(v^*)$.  Elliptic regularity now shows that $ u^*,v^* \in W^{2,\frac{3}{2}_-}$ and this is contained in $L^p$ for any $ p <\infty$.    One last iteration with $ \alpha = \frac{1}{4}$ shows that $ f'(u^*) g(v^*) \in L^\frac{7}{4}(\Omega)$ and after considering elliptic regularity and the Sobolev imbedding we have $u^*$ is bounded.  By symmetry we see $v^*$ is also bounded.

\hfill $ \Box$

\noindent\textbf{Proof of Theorem \ref{pppp}.}    Let $ (\lambda^*,\gamma^*) \in \Upsilon$ and let $(u^*,v^*)$ denote the extremal solution associated with $ (H)_{\lambda^*,\gamma^*}$.  We  let $ (u,v)$ denote a minimal solution on the ray $ \Gamma_\sigma$ where $ \sigma= \frac{\gamma^*}{\lambda^*}$.   Without loss of generality we can assume $ \lambda=\sigma=1$ to simplify the calculations.
  Note the assumption on $ f''$
 shows  there is some $ \xi>0$ such that $  \frac{f(u)}{u^{1+\xi}} \rightarrow \infty$ as $ u \rightarrow \infty$. Using this along with the fact that $ q>1$ and Remark \ref{rem1} shows that      $ \Delta u^*, \Delta v^*  \in L^1(\Omega)$.
Set $ \alpha:=f(u)-1$ and $ \beta:=(v+1)^t-1$ where $ 1 < t < t_+(q):=q+\sqrt{q(q-1)}$ into the stability inequality to obtain
 \begin{eqnarray*}
 && ( q - \frac{t^2}{2t-1} ) \int f'(u) (v+1)^{2t+q-1} +  \int (f(u)-1) f''(u) | \nabla u|^2
 \\&&+  2 \sqrt{q (q -1)} \int \sqrt{ f(u) f''(u)} (v+1)^{q-1} (f(u)-1) ( (v+1)^t-1)
 \\ & \le & 2 q \int f(u) f'(u) (v+1)^{q-1} + 2 q \int f'(u) (v+1)^{q +t-1}
 \end{eqnarray*}
   We label these integrals as $ I_i$ for $ 1 \le i \le 5$ from left to right.   The condition on $t$ ensures the coefficient in front $I_1$ is positive.  We can rewrite
\[ I_2 = \int q h_1(u) f(u) (v+1)^{q-1}, \qquad \mbox{ where \; \; \; \; } h_1(u)=\int_0^u (f(\tau)-1) f''(\tau) d \tau.\]
One easily sees that $ \frac{h_1(u)}{f'(u)} \rightarrow \infty$ as $ u \rightarrow \infty$.  Let $T$ be sufficiently large such that $ h_1(u) \ge 10 f'(u)$ for all $ u \ge T$.  Then one easily sees that
\[ \int f(u) f'(u) (v+1)^{q-1} \le \frac{1}{10} \int h_1(u) f(u) (v+1)^{q-1} + f'(T) \int f(u) (v+1)^{q-1}.\] We also have
\[ \int f'(u) (v+1)^{q+t-1} \le T^{q +t-1} \int f'(u) + \frac{1}{T^t} \int f'(u) (v+1)^{q +2t-1},\]  and so after combining the estimates we have
 \begin{eqnarray*}
&& \left( q - \frac{t^2}{2t-1} - \frac{2 q}{T^t} \right) \int f'(u) (v+1)^{q +2t-1}+ \frac{4 q}{5} \int h_1(u) f(u) (v+1)^{q-1}  \\&&+2 \sqrt{q (q -1)} \int \sqrt{ f(u) f''(u)} (v+1)^{q-1} (f(u)-1) ( (v+1)^t-1)
\\ &\le&  2 q f'(T) \int f(u) (v+1)^{q-1} + 2 q T^{q +t-1} \int f'(u).
 \end{eqnarray*}
    Passing to the limit shows this inequality holds with $ (u^*,v^*)$ in place of $ (u,v)$.    But these last two integrals are finite and hence we have an estimate provided the first coefficient is positive,  which is indeed the case provided we take $ T$ bigger if necessary.     Hence each of the following integrals is finite
\[ (i) \; \;  \int f'(u^*) (v^*+1)^{q +2t-1}, \quad (ii) \; \;  \int f'(u^*) f(u^*) (v^*+1)^{q-1}, \quad (iii) \; \; \int f(u^*)^\frac{3}{2} (v^*+1)^{q +t-1},\] for all $ 1 < t < t_+(q)$.

Now note that  $ -\Delta (f(u)) = - f''(u) | \nabla u|^2 + \lambda q f'(u) f(u) (v+1)^{q-1}$ and since $f$ is convex and since $ f'(u) f(u) (v+1)^{q-1}$ is uniformly bounded in $L^1(\Omega)$ along the ray $ \Gamma_\sigma$ shows that $ f(u) $ is uniformly bounded in $L^{ \frac{N}{N-2}_-} = L^{3_-}$; and hence $ f(u^*) \in L^{3_-}$.

  We now use (i) and (iii) to show that $u^*$ is bounded. To do so, set $0<\epsilon<3/2$ define $$\tau>\max\left\{ \frac{3}{2\epsilon}, \frac{3}{3/2-\epsilon} \right\} >1,  \ \ \alpha:=  \frac{3(\tau-1)}{2\tau}>0 \ \  \text{and} \ \ p=\frac{3}{2}+\epsilon.$$
From the definition of $\alpha$,  we have $\tau'\alpha=\frac{3}{2}$, where $\tau'$ and $\tau$ are conjugates meaning that $\frac{1}{\tau}+\frac{1}{\tau'}=1$.  It is straightforward to see that $(p-\alpha)\tau<3$ since $\tau > \frac{3}{2\epsilon}$,  and also $p\tau'<3$ since $\tau>\frac{3}{3/2-\epsilon}$. Elementary calculations show that the function $L(q):=1+\frac{t_+(q)}{q-1}$ is decreasing in $q>1$ and $L(q)>3$ for $q>1$.    As a result, $p\tau'<L(q)=1+\frac{t_+(q)}{q-1}$.  Therefore, $(q -1) p \tau'<  q  + t_+(q)-1$. From the H\"{o}lder's inequality we have
   \begin{eqnarray*}
\int f(u^*)^p (v^*+1)^{(q-1)p} &=&   \int f(u^*)^{p-\alpha} f(u^*)^{\alpha} (v^*+1)^{(q-1)p}
\\&\le&  \left( \int f(u^*)^{(p-\alpha) \tau} \right)^\frac{1}{\tau} \left( \int f(u^*)^\frac{3}{2} (v^*+1)^{ \tau' p (q-1) } \right)^\frac{1}{\tau'}
\end{eqnarray*}
 but the right hand side is finite and hence by elliptic regularity we have $ u^*$ is bounded.
  We now show that $v^*$ is bounded.   First note that we have  $ \int (v^*+1)^{2t+q-1} < \infty$ for any $ 1 < t < t_+(q)$ and hence one has $ \int (v^*+1)^{q +1} < \infty$.  Since $u^*$ is bounded this shows that $ v^* \in H_0^1(\Omega)$.   Now to complete the proof of $ v^*$ being bounded it is sufficient to show (since $u^*$ is bounded) that  $ (v^*+1) \in L^{(q-1)p}(\Omega)$ for some $ p >\frac{3}{2}$ but this easily follows after considering the above estimate.

\hfill $ \Box$

\noindent\textbf{Proof of Theorem \ref{radial}.}
  Step 1.   Let $(u,v)$ denote a smooth minimal solution of $(G)_{\lambda,\gamma}$ on the ray $ \Gamma_\sigma$ where $ \sigma:=\frac{\gamma^*}{\lambda^*}$.
Then taking a derivative of $(G)_{\lambda,\gamma}$ with respect to $r$ gives
 \begin{eqnarray}\label{pderadial}
\left\{ \begin{array}{lcl}
\hfill   -\Delta u_r +\frac{N-1}{r^2} u_r   &=& \lambda f''(u) g(v) u_r +\lambda f'(u)g'(v) v_r \qquad \text{for} \qquad 0<r<1, \\
\hfill -\Delta v_r+\frac{N-1}{r^2} v_r &=& \gamma f'(u) g'(v) u_r +\gamma f(u)g''(v) v_r   \qquad \text{for} \qquad  0<r<1.
\end{array}\right.
  \end{eqnarray}
Multiply the first and the second equations of (\ref{pderadial}) with $u_r \phi^2$ and $v_r \phi^2$ where  $ \phi  \in C^{0,1}(B_1) \cap H_0^1(B_1)$ gives
\begin{eqnarray}\label{integralradial}
\nonumber\int  |\nabla u_r|^2\phi^2 +\frac{1}{2} \nabla u_r^2\cdot\nabla\phi^2 +\frac{N-1}{r^2} u_r^2\phi^2   &=& \int \lambda f''(u) g(v) u_r^2\phi^2 +\lambda f'(u)g'(v) v_ru_r\phi^2  \\
\int  |\nabla v_r|^2\phi^2 +\frac{1}{2} \nabla v_r^2\cdot\nabla\phi^2 +\frac{N-1}{r^2} v_r^2\phi^2  &=& \int \gamma f'(u) g'(v) u_r v_r\phi^2+\gamma f(u)g''(v) v_r^2\phi^2
  \end{eqnarray}
  On the other hand, testing (\ref{gra}) on $\phi \to u_r\phi$ and $\psi \to v_r\phi$ where  $\phi$ is as above, we get
  \begin{eqnarray*}
\int f''(u)g(v) u_r^2\phi^2 +\int f(u)g''(v) v_r^2\phi^2+2 \int f'(u)g'(v)u_r v_r\phi^2\le \frac{1}{\lambda} \int |\nabla(u_r\phi)|^2+\frac{1}{\gamma} \int |\nabla(v_r\phi)|^2.
  \end{eqnarray*}
Expanding the right hand side we get
  \begin{eqnarray*}
\frac{1}{\lambda} \int  \left(|\nabla u_r|^2\phi^2+ u_r^2|\nabla \phi|^2 +\frac{1}{2} \nabla \phi^2\cdot\nabla u_r^2\right)  +\frac{1}{\gamma} \int \left( |\nabla v_r|^2\phi^2+ v_r^2|\nabla \phi|^2 +\frac{1}{2} \nabla \phi^2\cdot\nabla v_r^2 \right)
  \end{eqnarray*}
  Applying (\ref{integralradial}) the above will be
  \begin{eqnarray*}
&&\frac{1}{\lambda} \int   u_r^2|\nabla \phi|^2  + \frac{1}{\gamma} \int v_r^2|\nabla \phi|^2  -  \frac{N-1}{\lambda}\int u_r^2 \frac{\phi^2}{r^2}-  \frac{N-1}{\gamma}\int v_r^2 \frac{\phi^2}{r^2} \\&& + 2\int f'(u) g'(v) u_r v_r\phi^2+ \int f(u)g''(v) v_r^2\phi^2 +\int   f''(u) g(v) u_r^2\phi^2 .
  \end{eqnarray*}
 Therefore one obtains, after substituting $ r \phi$ for $ \phi$,
  \begin{equation}\label{radialstab}
  (N-1) \int \left( \frac{u_r^2}{\lambda} +\frac{v_r^2}{\gamma}  \right)\phi^2 \le \int \left( \frac{u_r^2}{\lambda} +\frac{v_r^2}{\gamma}  \right) |\nabla(r\phi)|^2,
  \end{equation}
   for all $ \phi \in C^{0,1}(B_1) \cap H_0^1(B_1)$.  Note that there is no $f$ and $g$ in this estimate.

   Step 2.    We show that (\ref{radialstab}) implies that $u^*,v^* \in H_0^1(B_1)$.  Firstly we argue there is some $ C_R>0$ such that for any $ 0<R<1$ we have $ \sup_{r>R} (u(r)+v(r)) \le C_R$ and $C_R$ is independent of $ \lambda$ (and hence the estimate also holds for $ u^*,v^*$).   To see this we first note that by Remark \ref{rem1} we have $ \| u\|_{L^1}, \|v\|_{L^1} \le C$ (uniformly in $ \lambda$) and since $u,v$ are radially decreasing we have the desired result otherwise we couldn't have the $L^1$ bound.     We now let $ 0 \le \phi \le 1$ be a smooth function supported in $ B_1$ with $ \phi=1$ on $B_\frac{1}{2}$.  Putting this into (\ref{radialstab}) and rearranging gives
   \begin{equation} \label{mos}
   (N-2) \int_{B_\frac{1}{2}} \frac{u_r^2}{\lambda} + \frac{v_r^2}{\gamma} \le C \int_{B_1 \backslash B_\frac{1}{2}} \frac{u_r^2}{\lambda} + \frac{v_r^2}{\gamma}.
   \end{equation}
    Now let $0 \le \psi \le 1$ denote  smooth function with $ \psi=0$ in $ B_\frac{1}{4}$ and $ \psi=1$ for $ |x|>\frac{1}{2}$.  Multiply $-\Delta u = \lambda f'(u) g(v)$ by $ u \psi^2$ and integrate by parts and use Young's inequality to arrive at
   \[ \int | \nabla u|^2 \psi^2 \le 2 \lambda \int f'(u) g(v) u \psi^2 + 4 \int u^2 | \nabla \psi|^2,\] and hence we have
   \begin{equation} \label{most}
   \int_{B_1 \backslash B_\frac{1}{2}} | \nabla u|^2  \le 2 \lambda \int_{B_1 \backslash B_\frac{1}{4}} f'(u) g(v) u + C \int_{B_1 \backslash B_\frac{1}{4}} u^2
   \end{equation} and we now use the pointwise bound to see that
   \[ \int_{B_1 \backslash B_\frac{1}{2}} | \nabla u|^2  \le C,\] where $C$ is independant of $ \lambda$.  Similarly we obtain the same estimate of $v$ and combining this with (\ref{mos}) we see that $ u,v$ are bounded in $H_0^1(B_1)$ indapendant of $ \lambda$ and hence $ u^*,v^* \in H_0^1(B_1)$.

Step 3.  Let $ (u,v)$  denote a minimal solution of $(G)_{\lambda,\gamma}$  on the $\Gamma_\sigma$ where $ \sigma:=\frac{\gamma^*}{\lambda^*}$.    For $0<r < \frac{1}{2}$ define $\phi$  to be the following test function
  \begin{eqnarray*}
\phi(t)= \left\{ \begin{array}{lcl}
\hfill   r^{-\sqrt{N-1}-1}   \quad && \text{if}\  0\le t\le r,\\
\hfill t^{-\sqrt{N-1}-1}   \quad && \text{if}\  r< t\le 1/2, \\
\hfill 2^{\sqrt{N-1}+2} (1-t)  \quad && \text{if}\  1/2< t\le 1.
\end{array}\right.
  \end{eqnarray*}  Putting $ \phi$ into  (\ref{radialstab}) gives
 \begin{eqnarray*}
 \int_{0}^{r} \left( \frac{u_r^2(t)}{\lambda} +\frac{v_r^2(t)}{\gamma}  \right) t^{N-1} dt \le C_N r^{2\sqrt{N-1}+2} \left( \frac{1}{\lambda} ||\nabla u||^2_{L^2(B_1\setminus \overline{B_{1/2}})} + \frac{1}{\gamma} ||\nabla v||^2_{L^2(B_1\setminus \overline{B_{1/2}})} \right),
 \end{eqnarray*} for all $ 0 <r< \frac{1}{2}$ and one easily extends this to all $ 0 <r<1$ by taking $ C_N$ bigger if necesary.
From this,  by simple calculations we get
  \begin{eqnarray*}
 \frac{1}{\sqrt{\lambda}} |u(r)-u(\frac{r}{2})|+\frac{1}{\sqrt\gamma} |v(r)-v(\frac{r}{2})| &\le& \int_{r/2}^{r}\left(\frac{1}{\sqrt\lambda}  |u_r(t)| + \frac{1}{\sqrt\gamma} |v_r(t)|\right) t^{\frac{N-1}{2}} t^{\frac{1-N}{2}} dt\\&\le&
\sqrt 2 \left(  \int_{r/2}^{r} \left( \frac{u_r^2(t)}{\lambda} +\frac{v_r^2(t)}{\gamma}  \right) t^{N-1} dt  \right)^{1/2} \left( \int_{r/2}^{r} t^{1-N} dt  \right)^{1/2}\\
& \le& C_N r^{\sqrt{N-1}+2-\frac{N}{2}} \left( \frac{1}{\sqrt\lambda} ||\nabla u||_{L^2(B_1\setminus \overline{B_{1/2}})} + \frac{1}{\sqrt\gamma} ||\nabla v||_{L^2(B_1\setminus \overline{B_{1/2}})} \right).
  \end{eqnarray*}
Let $0<r \le 1$. Then, there exist $m\in \mathbb N$ and $1/2<r_1\le 1$ such that $r =\frac{r_1}{2^{m-1}}$. Since $u,v$ are radial, we have $u(r_1)\le ||u||_{L^\infty(B_1\setminus \overline{B_{1/2}})} \le C_N ||u||_{H^1(B_1\setminus \overline{B_{1/2}})} $ and $v(r_1)\le ||v||_{L^\infty(B_1\setminus \overline{B_{1/2}})} \le C_N ||v||_{H^1(B_1\setminus \overline{B_{1/2}})} $.
\begin{eqnarray*}
\frac{1}{\sqrt{\lambda}} |u(r)| +\frac{1}{\sqrt{\gamma}} |v(r)| &\le& \frac{1}{\sqrt{\lambda}} |u(r_1)-u(r)|+\frac{1}{\sqrt{\gamma}} |v(r_1)-v(r)| +\frac{1}{\sqrt{\lambda}} |u(r_1)| + \frac{1}{\sqrt{\gamma}} |v(r_1)| \\&\le& \frac{1}{\sqrt{\lambda}} \sum_{i=1}^{m-1} \left|  u\left(\frac{r_1}{2^{i-1}}\right)-u\left( \frac{r_1}{2^i} \right) \right|+ \frac{1}{\sqrt{\gamma}} \sum_{i=1}^{m-1} \left|  v\left(\frac{r_1}{2^{i-1}}\right)-v\left( \frac{r_1}{2^i} \right) \right|
\\&&+ \frac{C_N}{\sqrt{\lambda}} ||u||_{H^1(B_1\setminus \overline{B_{1/2}})} + \frac{C_N}{\sqrt{\gamma}} ||v||_{H^1(B_1\setminus \overline{B_{1/2}})}
\\&\le& C_N \sum_{i=1}^{m-1} \left(\frac{r_1}{2^{i-1}}\right)^{-N/2+\sqrt{N-1} +2} \left( \frac{1}{\sqrt\lambda} ||\nabla u||_{L^2(B_1\setminus \overline{B_{1/2}})} + \frac{1}{\sqrt\gamma} ||\nabla v||_{L^2(B_1\setminus \overline{B_{1/2}})} \right)\\&&+ \frac{C_N}{\sqrt{\lambda}} ||u||_{H^1(B_1\setminus \overline{B_{1/2}})} + \frac{C_N}{\sqrt{\gamma}} ||v||_{H^1(B_1\setminus \overline{B_{1/2}})} \\&\le& C_N \left(  \sum_{i=1}^{m-1} \left(\frac{r_1}{2^{i-1}}\right)^{-N/2+\sqrt{N-1} +2} +1     \right) \left(\frac{1}{\sqrt{\lambda}} ||u||_{H^1(B_1\setminus \overline{B_{1/2}})} + \frac{1}{\sqrt{\gamma}} ||v||_{H^1(B_1\setminus \overline{B_{1/2}})} \right)
   \end{eqnarray*}
Note that the sign of ${\sqrt{N-1}+2-\frac{N}{2}}$ is crucial in getting estimates. Since  $\sqrt{N-1}+2-\frac{N}{2}=0$ if and only if $N=10$, this dimension is the critical dimension.  From the above, for any $0<r\le 1$ we get if $2\le N<10$,
 \begin{eqnarray*}
\frac{1}{\sqrt{\lambda}} |u(r)| +\frac{1}{\sqrt{\gamma}} |v(r)| \le C_N \left( \frac{1}{\sqrt\lambda} ||u||_{H^1(B_1\setminus \overline{B_{1/2}})} +\frac{1}{\sqrt\gamma} ||v||_{H^1(B_1\setminus \overline{B_{1/2}})} \right),
   \end{eqnarray*}
if $N=10$,
\begin{eqnarray*}
\frac{1}{\sqrt{\lambda}} |u(r)| +\frac{1}{\sqrt{\gamma}} |v(r)|  \le C_N (1+|\log r|) \left( \frac{1}{\sqrt\lambda} ||u||_{H^1(B_1\setminus \overline{B_{1/2}})} +\frac{1}{\sqrt\gamma} ||v||_{H^1(B_1\setminus \overline{B_{1/2}})} \right),
   \end{eqnarray*}
   if $N>10$,
\begin{eqnarray*}
 \frac{1}{\sqrt{\lambda}} |u(r)| +\frac{1}{\sqrt{\gamma}} |v(r)| \le C_N r^{-\frac{N}{2} +\sqrt{N-1}+2}  \left( \frac{1}{\sqrt\lambda} ||u||_{H^1(B_1\setminus \overline{B_{1/2}})} +\frac{1}{\sqrt\gamma} ||v||_{H^1(B_1\setminus \overline{B_{1/2}})} \right).   \end{eqnarray*}  Passing to limits we obtain the desired estimates for $ u^*,v^*$.

\hfill $ \Box$

 \section{Explicit nonlinearities}
 We now examine the case of polynomial nonlinearities and for this we recall the definition $ t_+(p)=  p+ \sqrt{p(p-1)}$
 and note that $ t_+$ is increasing on $[1,\infty)$.   We begin with the gradient system.
\begin{thm} \label{grade} Let $ f(u)=(u+1)^p,  g(v)=(v+1)^q$ and suppose $ p,q>2$.
\begin{enumerate} \item Let $(\lambda^*,\gamma^*) \in \Upsilon$.  Then the associated extremal solution of  $ (G)_{\lambda^*,\gamma^*}$ is bounded provided
\begin{equation} \label{gradp} \frac{N}{2} < 1 + \frac{2}{p+q-2} \max\{ t_+(p-1),t_+(q-1)\}.
\end{equation}

\item Let $(\lambda^*,\gamma^*) \in \Upsilon$ and define $ I_{p,q,\lambda,\gamma}(t) := p+q-1-\frac{t^2}{2t-1}+\frac{2pq}{p+q} \left(  (\frac{\gamma q}{\lambda p})^{t+q-1} -1  \right)$.  Let  $ t_0:= \max\{ t_+(p-1),t_+(q-1)\}$ and
suppose that  $ I_{p,q,\lambda^*,\gamma^*}(t_0)>0$ and $ I_{q,p,\gamma^*,\lambda^*}(t_0)>0$.   The map $ t \mapsto \min\{I_{p,q,\lambda^*,\gamma^*}(t), I_{q,p,\gamma^*,\lambda^*}(t) \}$
is decreasing and has a root in $(t_0,\infty)$, which we denote by $T$.  Suppose
\begin{equation} \label{zaza}
\frac{N}{2} < 1+\frac{2}{p+q-2} T.
\end{equation}    Then the associated extremal solution of  $ (G)_{\lambda^*,\gamma^*}$ is bounded.
\end{enumerate}
\end{thm}
\begin{remark}  Note that the condition on $ t_0$ from the second part of Theorem  \ref{grade} is really a condition on how close the parameters $ (\lambda^*,\gamma^*)$ are to the ``diagonal'' given by $ \lambda^* p = \gamma^* q$.   On the diagonal one trivially sees the condition is satisfied.  Some algebra shows that the condition is satisfired  provided $ (\lambda^*,\gamma^*)$ lie within the cone
\begin{eqnarray*} \left( 1- \frac{p+q}{2pq} \min(p,q) \right)^\frac{1}{t_0+q-1}  < \frac{\gamma^* q}{\lambda^*p}  < \left( 1- \frac{p+q}{2pq} \min(p,q) \right)^{-\frac{1}{t_0+p-1}}.
\end{eqnarray*}
 \end{remark}

\begin{thm} \label{hamil}    Let $ f(u)=(u+1)^p$, $ g(v):=(v+1)^q$ with $ p,q>1$. \item   Suppose
\[ N < \min \left\{ 4 + \frac{2}{p-1} + 2 \sqrt{ \frac{p}{p-1}},  4 + \frac{2}{q-1} + 2 \sqrt{ \frac{q}{q-1}} \right\},\]
and $(\lambda^*,\gamma^*) \in \Upsilon$.  Then the associated extremal solution of  $ (H)_{\lambda^*,\gamma^*}$ is bounded.
\end{thm}
\begin{remark} Note that $ p \mapsto 4 + \frac{2}{p-1} + 2 \sqrt{ \frac{p}{p-1}}$ is decreasing and goes to $ 6$ as $ p \rightarrow \infty$.  Hence for $ N \le 6$ we see all extremal solutions are bounded for any $ p,q$.  As in the case of $(G)_{\lambda,\gamma}$ one can obtain better results provided they restrict the range of the parameters $ (\lambda,\gamma)$ to a  certain cone with axis given by  $\frac{\gamma}{\lambda} = \frac{q}{p}$,  we omit the details.

\end{remark}

We begin with some pointwise comparison results.
\begin{lemma}  \label{pointwise}
Let $ f(u)=(u+1)^p$ and $  g(v)=(v+1)^q$ where  $ p,q>1$.
\begin{enumerate} \item Suppose that $(u,v)$ is a smooth solution of $(G)_{\lambda,\gamma}$ where $ \lambda p\ge \gamma q$.   Then  $ v\le u \le \frac{\lambda p}{\gamma q} v$.

\item  Suppose  $(u,v)$ is the smooth minimal solution of $(H)_{\lambda,\gamma}$  where $ q \lambda \ge \gamma p$.
 Then  $ p \gamma u \ge q \lambda v$.
\end{enumerate}
 \end{lemma}

\noindent\textbf{Proof:} (1) Subtracting two equations of $(G)_{\lambda,\gamma}$ we get
\begin{eqnarray*}
-\Delta (u-v)  &=& (1+u)^{p-1}(1+v)^{q-1}(\lambda p (1+v) - \gamma q (1+u))\\&\ge& \gamma q  (1+u)^{p-1}(1+v)^{q-1} (v-u)
\end{eqnarray*}
multiply both sides of the above with $(u-v)_-$ to get $ \int |\nabla (u-v)_{-}|^2 \le 0 $
and therefore $v \le u$.  Now, multiply the second equation of $(G)_{\lambda,\gamma}$  with $\frac{\lambda p}{\gamma q}$ and again subtract two equations to get
\begin{eqnarray*}
-\Delta (u-  \frac{\lambda p}{\gamma q} v)  = \lambda p (1+u)^{p-1}(1+v)^{q-1}(v - u) \le 0
\end{eqnarray*}
From maximum principle we get $u \le \frac{\lambda p}{\gamma q} v$.
\\
(2) Set $K(x):= (u+1)^{p-1} (v+1)^{q-1}$.    First note that
\[ L(u-v):=-\Delta (u-v) -\gamma p K(x) (u-v) =K(x) ( (\lambda q - \gamma p) u + \lambda q - \gamma p ),\]  and note that the right hand side is nonnegative.  If we can show that $L$ satisfies the maximum principle then we'd have $ u-v \ge 0$.

We now assume that  $ (u,v)$ is the smooth  minimal solution of $(H)_{\lambda,\gamma}$ and additional we assume that $ (\lambda,\gamma) \in \mathcal{U} \backslash \Upsilon$.  By  Theorem A there   is some $ \eta  \ge 0 $ and $ \psi >0$ such that
\begin{equation} \label{pzpz}
 -\Delta \psi - q p \gamma K(x) \psi \ge \eta \psi.
\end{equation}    Since $(\lambda,\gamma) \notin \Upsilon$ one can infact show that $ \eta >0$. Hence  the linear operator on the left satisfies the maximum principle.  Since $ q >1$ we see that $L$ must also satisfy the maximum principle and hence $ u \ge v$.    In the case where $ (\lambda,\gamma) \in \mathcal{U} \cap \Upsilon$ we pass to the limit along the fixed parameter ray through $(0,0)$ and $ (\lambda,\gamma)$ and use the above result.    Hence we have shown that $ u \ge v$.
  Now set $ t:= \frac{\lambda q}{\gamma p}$ and then note that
\[ -\Delta (u-tv) = K(x) \lambda q ( (u+1)-(v+1)) \ge 0\] and hence $ u \ge t v$.

\hfill $ \Box$

Before we prove Theorem \ref{grade} we need a general energy estimate.
\begin{lemma}\label{stabpol}
 Let $ f(u)=(u+1)^p$ and $  g(v)=(v+1)^q$.  Suppose that $(u,v)$ is a semi-stable solution of $(G)_{\lambda,\gamma}$ with  $ s,t \in \IR \backslash \{ \frac{1}{2} \}$.  Then
 \begin{eqnarray*}
&&  p\left(  p-1-\frac{t^2}{2t-1}  \right) \int (1+u)^{2t+p-2} (1+v)^q +q \left(  q-1-\frac{s^2}{2s-1}  \right) \int (1+v)^{2s+q-2} (1+u)^p\\&&+2pq \int (1+u)^{t+p-1} (1+v)^{s+q-1} +p(p-1)\int (1+u)^{p-2}(1+v)^{q} + q(q-1)\int (1+u)^p(1+v)^{q-2}
\\&&+
p \frac{t^2}{2t-1} \int (1+u)^{p-1} (1+v)^q +q \frac{s^2}{2s-1} \int (1+u)^{p} (1+v)^{q-1}
\\&&
\le  2p(p-1)\int (1+u)^{t+p-2} (1+v)^q + 2pq \int (1+u)^{t+p-1} (1+v)^{q-1} \\&&+ 2q(q-1)\int (1+u)^{p} (1+v)^{s+q-2}+  2pq\int (1+u)^{p-1} (1+v)^{s+q-1}
  \end{eqnarray*}
 \end{lemma}

\begin{proof}  This is an application of Lemma \ref{stabb}. Take $\phi:=(1+u)^{t}-1$ and $\psi:=(1+v)^s-1$ in (\ref{gra}), then we have
\begin{eqnarray}\label{ineq}
\nonumber && p(p-1) \int (1+u)^{p-2} (1+v)^q \left( (1+u)^t-1 \right)^2 + q(q-1) \int (1+v)^{q-2} (1+u)^p \left( (1+v)^s-1 \right)^2\\\nonumber && + 2 pq\int (1+u)^{p-1} (1+v)^{q-1} \left( (1+u)^t-1 \right) \left( (1+v)^s-1 \right) \\ & &\le \frac{t^2}{\lambda} \int |\nabla u|^2 (1+u)^{2t-2} + \frac{s^2}{\gamma} \int |\nabla v|^2  (1+v)^{2s-2}
\end{eqnarray}
Multiply the first and the second equation of $(G)_{\lambda,\gamma}$ with $(1+u)^{2t-1}-1$ and $(1+v)^{2s-1}-1$, respectively, to get
$$(2t-1)\int |\nabla u|^2(1+u)^{2t-2}=\lambda p \int  (1+u)^{2t+p-2}(1+v)^q-\lambda p \int  (1+u)^{p-1}(1+v)^q $$ and
$$(2s-1) \int  |\nabla v|^2(1+v)^{2s-2}=\gamma q \int (1+v)^{2s+q-2}(1+u)^p-\gamma q \int (1+v)^{q-1}(1+u)^p. $$
 Using these identities and (\ref{ineq}) finishes the proof.

\end{proof}

\noindent\textbf{Proof of Theorem \ref{grade}.}   (1) Let $ (\lambda^*,\gamma^*) \in \Upsilon$ and let $(u,v)$ denote a smooth minimal solution on the ray $ \Gamma_\sigma$ where $ \sigma:= \frac{\gamma^*}{\lambda^*}$.

  Let $ 1 < t<t_+(p-1)$ and $ 1<s<t_+(q-1)$ in Lemma \ref{stabpol} to arrive at an inequality of the form
\[ \int (u+1)^{2t+p-2} (v+1)^q + \int (u+1)^p (v+1)^{2s+q-2} \le C_{t,s}.\] First note that
\[ \int | \nabla u|^2  \le p \lambda^* \int (u+1)^p (v+1)^q \le C_{t,s},\] provided
$ p < 2t_+(p-1) +p-2$ but this holds since $ p>1$ and by passing to the limit we see that $ u^* \in H_0^1(\Omega)$.   We similarly show that $ v^* \in H_0^1(\Omega)$.

 Without loss of generality assume that $ p \ge q$ and hence $ t_+(q-1) \le t_+(p-1)$ and so we have
\[ \int (u+1)^{2t+p+q-2} \le C_{t},\] for all $ 1 < t < t_+(p-1)$.   We now re-write the equation as
\[ \frac{ -\Delta u^*}{\lambda^*} = c(x) u^* + p (v^*+1)^q, \] where
\[ 0 \le c(x)= p \frac{  \left( (u^*+1)^{p-1}-1 \right)}{u^*} (v^*+1)^q \le C (u^*+1)^{p+q-2}.\]  We now apply regularity theory to see that $u^*$ is bounded provided $ c(x), (v^*+1)^q \in L^T$ for some $ T>\frac{N}{2}$.  But this holds provided
\[  (p+q-2) \frac{N}{2} < 2t_+(p-1) +p+q-2,\] which is the desired result.  To see $v^*$ is bounded we now use the pointwise comparison between $u$ and $v$ and pass to the limit along the ray $ \Gamma_\sigma$.  \\

(2) In (1) we only used the first two integrals from Lemma \ref{stabpol} to obtain estimates.  In this part we also use the third integral.
Let $ (\lambda^*,\gamma^*) \in \mathcal{U} $ and let $ (u,v)$ denote the a minimal solution on the ray $ \Gamma_\sigma$, where $ \sigma:=\frac{\gamma^*}{\lambda^*}$.    The exact proof depends on the sign of $ \lambda^* p - \gamma^* q$ and we suppose that $ \lambda^* p \ge \gamma^* q$.    Let $ t_0 <t < T$ and so $ I_{p,q,\lambda^*,\gamma^*}(t), I_{q,p,\lambda^*,\gamma^*}(t)>0$ and $ p-1-\frac{t^2}{2t-1}, q-1- \frac{t^2}{2t-1}  <0$.   We now set $ s=t$ and examine the estimate from Lemma \ref{stabpol}.  Note the coefficients in front of the first two integrals are negative and the coefficient in front of the third integral is positive.  The other integrals on the left are lower order terms which we drop.     Now note that $ u \ge v$ and so we can replace, since the coefficients are negative,  the $u$'s in the first two integrals from the estimate in Lemma \ref{stabpol} with $ v$'s.
  In the third integral we use the fact that $ \frac{\gamma q}{\lambda p} (u+1) \le v+1$. Writing this all out and then again using the fact that we can compare $u$ and $v$, one can see that the following is a lower bound for the left-hand side of the integral estimate given by Lemma \ref{stabpol},
 \begin{eqnarray*}
   \left(p\left(p-1-\frac{t^2}{2t-1}\right)  + q\left(q-1- \frac{t^2}{2t-1}\right) +2pq  \left( \frac{\gamma q}{\lambda p}\right)^{t+q-1}\right) \int (1+u)^{2t+p+q-2} \\+
(p+q)\left( p+q-1-\frac{t^2}{2t-1} +\frac{2pq}{p+q} \left(  \left( \frac{\gamma q}{\lambda p}\right)^{t+q-1}-1 \right) \right)\int (1+u)^{2t+p+q-2}.
 \end{eqnarray*}
Combining everything gives an estimate of the form
 \[ I_{p,q, \lambda^*,\gamma^*}(t)  \int (1+u)^{2t+p+q-2}   \le C_{p,q,\lambda^*,\gamma^*} \int (1+u)^{t+p+q-2}.\]
Since $ I_{p,q,\lambda^*,\gamma^*}(t)>0$ we have an estimate.   We now proceed exactly as in the first part. We rewrite the equation in the alternate form and we then require that
 \[ (p+q-2) \frac{N}{2} < 2t +p +q-2,\] for some $ t_0< t $ where $I_{p,q, \lambda^*,\gamma^*}(t) >0$.

\hfill $ \Box$

\begin{lemma} \label{shit} Let $ (\lambda^*,\gamma^*) \in \Upsilon$ and let $ (u,v)$ denote a minimal solution of $(H)_{\lambda,\gamma}$ on the ray $ \Gamma_\sigma$ where $ \sigma = \frac{\gamma^*}{\lambda^*}$.  Then for $ 1 <t<t_+(p)$ and $ 1 < \tau < t_+(q)$ we have
\begin{eqnarray}
 \int (u+1)^{2t+p-1} (v+1)^{q-1} &\le& C, \\
 \int (u+1)^{p-1} (v+1)^{2 \tau + q-1} &\le& C,
\end{eqnarray}
where $ C$ is uniform on the ray $ \Gamma_\sigma$. These two inequalities and an application of  the Cauchy-Schwarz inequality gives
 \[ \int (u+1)^{p+t-1} (v+1)^{q+\tau-1} \le C.\]
\end{lemma}

\begin{proof} Set $ \phi:=(u+1)^t-1$ and $ \psi:=(v+1)^\tau-1$ and put these into the stability inequality given by (\ref{twist}) to arrive at   an inequality of the form
\begin{eqnarray*}
&& q ( p - \frac{t^2}{2t-1}) \int (u+1)^{2t+p-1} (v+1)^{q-1} +  p( q - \frac{\tau^2}{2 \tau-1}) \int (u+1)^{p-1}(v+1)^{q+2 \tau-1}
 \\&& + 2 \sqrt{p(p-1)q (q-1)} \int(u+1)^{p+t-1} (v+1)^{q+\tau-1}
 \le    C(p,q) \int (u+1)^{p+t-1} (v+1)^{q-1}  \\
 && + C(p,q) \int (u+1)^{p-1} (v+1)^{q+\tau-1}.
 \end{eqnarray*}  Note that for the given choices of $ t, \tau$ the coefficients on the left are positive.   One now easily sees that the terms on the right are lower order terms and hence we obtain the desired estimates after some standard calculations.

\end{proof}

\noindent\textbf{Proof of Theorem \ref{hamil}.}  Without loss of generality we suppose that $ \lambda^* q \ge \gamma^* p$.   Let $ (u,v)$ denote a minimal solution on the ray $ \Gamma_\sigma$ where $ \sigma:= \frac{\gamma^*}{\lambda^*}$.   Note that we have
$ u \ge  \frac{ \lambda^* q}{\gamma^* p} v >v$.   We first show that $ u^* \in H_0^1$.  First note that
\[ \int | \nabla u|^2 = \lambda q \int (u+1)^p u (v+1)^{q-1}, \]
along the ray $ \Gamma_\sigma$ and the right hand side is uniformly bounded   provided
\[ p+1 < p-1 + 2 t_+(p),\] which is the case, for any $ p>1$ and dimension $N$,  after considering the estimates from Lemma \ref{shit}.
We now rewrite the equation for $u^*$ as
\[ -\Delta u^* = \lambda q \left( \frac{ (u^*+1)^{p}-1}{u^*} \right) (v^*+1)^{q-1} u^* + \lambda q (v^*+1)^{q-1},\] and to show $u^*$ is bounded it is sufficient to show that
$ (u^*+1)^{p-1} (v^*+1)^{q-1} \in L^r$ for some $ r > \frac{N}{2}$.   Using Lemma \ref{shit} one sees this is the case   provided
\[ \frac{N}{2} (p-1) < p-1 + t_+(p), \qquad \frac{N}{2}(q-1) < q -1 + t_+(q).\]   So we have shown that $u^*$ is bounded and we now use the fact that $u^* \ge v^*$ to see the same for $ v^*$.

\hfill $ \Box$

\end{document}